\documentclass[11pt, oneside]{amsart}

\usepackage{pinlabel}

\usepackage{amsmath, amsfonts, amssymb}
\usepackage{amsopn}
\usepackage{enumerate}
\usepackage{ifthen, color, hyperref, multicol}
\usepackage[cmtip,arrow]{xy}
\usepackage{subfig}

\usepackage{enumitem}

\usepackage[T1]{fontenc}

\newcommand{\showcomments}{yes}

\newsavebox{\commentbox}
\newenvironment{com}%
{\ifthenelse{\equal{\showcomments}{yes}}%
	{\footnotemark
		\begin{lrbox}{\commentbox}
			\begin{minipage}[t]{1.in}\raggedright\sffamily\tiny
				\footnotemark[\arabic{footnote}]}
			{\begin{lrbox}{\commentbox}}}%
			{\ifthenelse{\equal{\showcomments}{yes}}%
				{\end{minipage}\end{lrbox}\marginpar{\usebox{\commentbox}}}
		{\end{lrbox}}}

\newcounter{ax}
\newtheorem{thm}{Theorem}[section]

\newtheorem{lem}[thm]{Lemma}
\newtheorem{prop}[thm]{Proposition}

\newtheorem{cor}[thm]{Corollary}
\newtheorem{conj}[thm]{Conjecture}

\newtheorem*{thmi}{Theorem}

\newtheorem*{quesi}{Question}

\theoremstyle{definition}
\newtheorem{defn}[thm]{Definition}

\newtheorem{claim*}{Claim}

 \DeclareMathOperator{\link}{Lk}

\makeatletter

\newcommand{\Rmnum}[1]{\mathbf{{\expandafter\@slowromancap\romannumeral #1@}}}

 \makeatother

\let\oldmarginpar\marginpar
\renewcommand\marginpar[1]{\-\oldmarginpar[\raggedleft\footnotesize #1]%
	{\raggedright\footnotesize #1}}

\newcommand\ds{\displaystyle}

\newcounter{enumitemp}

\def\G{{\mathcal G}}
\newcommand{\CFS}{\ensuremath{\mathcal{CFS}}}

\newcommand{\E}{\mathbb E}
\newcommand{\Pb}{\mathbb P}

\newcommand{\gspan}[1]{\left\langle #1\right\rangle}
\newcommand{\Cayley}[2]{\operatorname{Cay}\left(#1, #2\right)}

\definecolor{figblue}{cmyk}{1,0,0,0}
\definecolor{figred}{cmyk}{0,1,1,0}
\definecolor{figgreen}{cmyk}{.75,0,1,0}
\definecolor{figblue2}{cmyk}{1,1,0,0}
\definecolor{figgreen2}{cmyk}{1,0,1,0}

\long\def\Restate#1#2#3#4{
	\medskip\par\noindent
	{\bf #1 \ref{#2} #3} {\it #4}\par\medskip }

\usepackage{tikz}
\usetikzlibrary{arrows.meta}
\usetikzlibrary{quotes, angles, decorations, decorations.markings}

\title{Morse subgroups and boundaries of random right-angled Coxeter groups}

\author[T. Susse]{Tim Susse}
\address{Bard College at Simon's Rock, Great Barrington, Mass., USA}
\email{tsusse@simons-rock.edu}


\thanks{}

\numberwithin{thm}{section}
\numberwithin{equation}{section}
\begin{document}
	\maketitle
	
	\begin{abstract} We study Morse subgroups and Morse boundaries of random right-angled Coxeter groups in the Erd\H{o}s--R\'enyi model. We show that at densities below  $\left(\sqrt{\frac{1}{2}}-\epsilon\right)\sqrt{\frac{\log{n}}{n}}$ random right-angled Coxeter groups almost surely have Morse hyperbolic surface subgroups. This implies their Morse boundaries contain embedded circles. Further, at densities above $\left(\sqrt{\frac{1}{2}}+\epsilon\right)\sqrt{\frac{\log{n}}{n}}$ we show that, almost surely, the hyperbolic Morse special subgroups of a random right-angled Coxeter group are virtually free. 
	
	We also apply these methods to show that for a random graph $\Gamma$ at densities below\\ $(1-\epsilon)\sqrt{\frac{\log{n}}{n}}$, $\square(\Gamma)$ almost surely contains an isolated vertex. This shows, in particular, that at densities below $(1-\epsilon)\sqrt{\frac{\log{n}}{n}}$ a random right-angled Coxeter group is almost surely not quasi-isometric to a right-angled Artin group \end{abstract}
	
	\section{Introduction}
	Given a simplicial graph $\Gamma$ with vertex set $V$, and edge set $E$ we can form the right-angled Coxeter group with presentation:
	
$$W_\Gamma := \gspan{V \mid v^2=1\ \forall v\in V, vw=wv \iff \{v,w\}\in E}.$$

That is, the group has order two generators whose commuting relations are determined by the edges of $\Gamma$. Much work on right-angled Coxeter groups focuses on the connections between the combinatorics of $\Gamma$ and the geometry and topology of $W_\Gamma$.

In this paper, we will be particularly interested in the Morse boundary of right-angled Coxeter groups, and the relationship between right-angled Coxeter groups and right-angled Artin groups. These groups are defined similarly to right-angled Coxeter groups: given a graph $\Gamma$ the corresponding right-angled Artin group has presentation:

$$A_\Gamma = \gspan{V \mid vw=wv \iff \{v,w\}\in E}.$$  

One direction of this relationship was cemented by the following celebrated theorem of Davis and Januszkiewicz.

\begin{thmi}\cite{DavisJanuszkiewicz} Given a graph $\Gamma$, there exists a graph $\Lambda$ so that $A_\Gamma$ is isomorphic to a finite index subgroup of $W_\Lambda$.\end{thmi}

It is well-known that not every right-angled Coxeter group has a finite index right-angled Artin subgroup, or is even quasi-isometric to a right-angled Artin group (see \cite{BehrstockCharney, DaniThomas:divcox}, showing that right-angled Artin and Coxeter groups have different possible divergence functions). This leads to the following folk question.

\begin{quesi} Which right-angled Coxeter groups are virtually right-angled Artin?\end{quesi}

Much work has been done recently on this question, focusing on the Morse boundary of right-angled Coxeter groups, introduced in \cite{CharneySultan, Cordes:Morse}. In \cite{CCS:Morse}, it is shown that right-angled Artin groups have totally disconnected Morse boundaries, while in \cite{Behrstock:Morse}, an example of a right-angled Coxeter group with quadratic divergence containing an embedded $S^1$ in its Morse boundary is given. To do this, Behrstock found a special subgroup of the right-angled Coxeter group which was virtually a hyperbolic surface group and Morse in the sense of \cite{Tran:quasiconvex}.

As in \cite{CharneyFarber, BHS:Coxeter, BFRHS, BFRS}, here we study the geometry of \emph{random} right-angled Coxeter groups, using the Erd\H{o}s--R\'enyi model of random graphs. In this model, we are given a function $p\colon \mathbb{N}\to (0,1)$, and generate a graph on $n$ vertices by declaring that any pair of vertices is joined by an edge with probability $p(n)$, independently of all other pairs. Using this model we build on Behrstock's work, giving a sharp threshold for one-ended hyperbolic Morse special subgroups in right-angled Coxeter groups.

\Restate{Theorem}{thm:noRAAG}{}{Let $\lambda=\sqrt{\ds\frac{1}{2}}$ and let $\epsilon>0 $. If $np\to \infty$ and  $p(n)<(\lambda-\epsilon)\sqrt{\ds\frac{\log{n}}{n}}$, then a.a.s. $W_\Gamma$ contains a Morse hyperbolic surface subgroup. In particular, the Morse boundary of $W_\Gamma$ contains an embedded copy of $S^1$ and is not totally disconnected.}

\Restate{Corollary}{cor:specialstable}{}{Let $\epsilon>0, \lambda=\sqrt{\ds\frac{1}{2}}$ and let $p(n)>(\lambda+\epsilon)\sqrt{\frac{\log{n}}{n}}$. Then for $\Gamma\in\G(n,p)$, a.a.s. every hyperbolic Morse special subgroup in $W_\Gamma$ is virtually free.}

Expanding our examination to consider a wider variety of Morse subgroups, we also obtain the following insight into the folk question above.

\Restate{Corollary}{cor:noRAAG}{}{Let $\epsilon>0$ if $p(n)<\left(1-\epsilon\right)\sqrt{\ds\frac{\log{n}}{n}}$ and $np\to \infty$, then a.a.s. a right-angled Coxeter group at density $p(n)$ is not quasi-isometric to a right-angled Artin group.}

 In our proofs we must assume that $np\to \infty$. When $np\to 0$, then a.a.s. a random graph is a forest, therefore the corresponding right-angled Coxeter group is virtually free (and hence virtually right-angled Artin). When $np\to c$ with $0<c<\infty$, then  $\Gamma$ contains a cycle with positive probability; however, if $\Gamma$ contains a $k$--cycle, with $k\ge 5$ then the group contains a Morse hyperbolic surfaces subgroup and is not quasi-isometric to a right-angled Artin group. If the only cycles are $3$-- and $4$--cycles, noting that at this density, a.a.s. no two cycles intersect, $\Gamma$ is a tree with some vertices replaced by squares and triangles. It's not hard to see that the corresponding group is virtually a free products of several copies of $\mathbb{Z}^2$ and a free group. Thus, at densities below $(1-\epsilon)\sqrt{\frac{\log{n}}{n}}$, a random right-angled Coxeter group is quasi-isometric to a right-angled Artin group if and only if it is virtually $\left(\ast_{i=1}^k \mathbb{Z}^2\right)\ast F_r$ for some $k,r\ge 0$.

 Recall from \cite{DaniThomas:divcox} that given a graph $\Gamma$, the square-graph of $\Gamma$, denoted by $\square(\Gamma)$, is the graph whose vertices are in one-to-one correspondence with the induced $4$--cycles (squares without diagonals) in $\Gamma$, and two vertices are joined by an edge if and only if the intersection of the two corresponding $4$--cycles contains a pair of non-adjacent vertices. A graph is $\CFS$ if  $\square(\Gamma)$ has a connected component $C$ so that every vertex of $\Gamma$ is contained in some square corresponding to a vertex in $C$.

 We can also use the methods of Theorem~\ref{thm:noRAAG} to show the following, interesting Corollary.
 
 \Restate{Corollary}{cor:disconnect}{}{Let $\epsilon>0$, $p(n)<(1-\epsilon)\sqrt{\frac{\log{n}}{n}}$ and $np\to \infty$ , then for $\Gamma\in\G(n,p)$ a.a.s. $\square(\Gamma)$ contains an isolated vertex. In particular, $\square(\Gamma)$ is disconnected.}
 
 This is particularly interesting as the Corollary mimics the first half of the classic proof of Erd\H{o}s and R\'enyi from \cite{ErdosRenyi1} on the connectedness of random graphs. This also furthers the work of Behrstock, Falgas-Ravry, Hagen and the author \cite{BFRHS, BFRS}, analyzing the square graph in the Erd\H{o}s--R\'enyi model.
 
 In \cite{DaniThomas:divcox, Levcovitz:CFSdiv}, it is shown that $W_\Gamma$ has quadratic divergence if and only if $\Gamma$ is a $\CFS$ graph. Further in \cite{BFRHS, BFRS}, the author, with Behrstock, Falgas-Ravry and Hagen determined that the threshold for $\CFS$ in random graphs is $\sqrt{\sqrt{6}-2}n^{-1/2}$, which is much smaller than $\sqrt{\frac{\log{n}}{n}}$. The above results (along with Corollary~\ref{cor:noRAAG} below) yields the following.
 
 \begin{cor} For $\epsilon>0$, if $(\sqrt{\sqrt{6}-2}+\epsilon)n^{-1/2}<p(n)<(1-\epsilon)\sqrt{\frac{\log{n}}{n}}$, a random graph $\Gamma\in\G(n,p)$ a.a.s. is $\CFS$, but contains a Morse $k$--cycle with $k\ge 4$. In particular, $W_\Gamma$ has quadratic divergence and is \emph{not} quasi-isometric to a right-angled Artin group.
 
Moreover, for $\left(\sqrt{\frac{1}{2}}+\epsilon\right)\sqrt{\frac{\log{n}}{n}} < p(n) < \left(1-\epsilon\right)\sqrt{\frac{\log{n}}{n}}$, $W_\Gamma$ has quadratic divergence, all hyperbolic Morse special subgroups are virtually free, and $W_\Gamma$ is \emph{not} quasi-isometric to a right-angled Artin group. \end{cor}

  In particular, there are infinitely many right-angled Coxeter groups with quadratic divergence that are not quasi-isometric to any right-angled Artin group, which was previously suggested but unknown. Indeed, there are infinitely many right-angled Coxeter groups with quadratic divergence whose only hyperbolic Morse special subgroups are virtually free, but the group is still not quasi-isometric to a right-angled Artin group. This phenomenon was first oberseved \cite{GKLS}, where a single example is given. This suggests that the problem of determining precisely when a right-angled Coxeter groups is quasi-isometric to a right-angled Artin groups is incredibly delicate (\emph{c.f.} \cite{NT:RACG}).

\section*{Acknowledgements} Thanks to Annette Karrer for reading drafts of this paper and suggesting edits which clarified the statements of the main results. This work was supported by NSF grant DMS-2137608.
 
	\section{Preliminaries and notation}

\subsection{Graph Theory Background}
	\begin{defn} Given a set $S$, by $S^{(2)}$ we mean the set of subsets of $S$ of size precisely $2$, in other words, the set of unordered pairs of distinct elements of $S$.
	A simplicial graph $\Gamma$ is a pair $\Gamma = (V,E)$, were $V$ is the set of \emph{vertices} and $E\subset V^{(2)}$ is the set of \emph{edges}. If ${v,w}\in E$, then we say that the vertices $v$ and $w$ are \emph{adjacent} or \emph{joined by an edge}.\end{defn}

We will commonly confuse $\Gamma$ with its vertex set $V$, and say $v\in \Gamma$ to mean that $v\in V$.

	\begin{defn} Let $\Gamma = (V,E)$ be a graph. A subgraph of $\Lambda\subseteq\Gamma$ is a pair $(V(\Lambda), E(\Lambda))$, where $V(\Lambda)\subseteq V$ and $E(\Lambda)\subseteq E$. We say that $\Lambda$ is \emph{induced} if $E(\Lambda)=E\cap \Lambda^{(2)}$.
		
	Given a vertex $v\in V$ the \emph{link} of $v$ in $\Gamma$, denoted $\link(v)$ or $\link_\Gamma(v)$, is the subgraph of $\Gamma$ induced by the vertices adjacent to $v$.

A $k$--cycle (or $k$--gon) in $\Gamma$ is a subgraph $C$ where $V(C)=\{v_1, \ldots, v_k\}$ and $E(C)=\{\{v_i, v_{i+1}\}: 1\le i\le k-1\}\cup\{v_k, v_1\}$\end{defn}

	We will be interested in 5--cycles in a random graph $\Gamma \in \G(n,p)$. An induced $k$--cycle in a graph $\Gamma$ corresponds to a subgroup of $W_\Gamma$ which is virtually a surface group. If $k=4$, this subgroup is virtually $\mathbb{Z}^2$, while if $k\ge 5$, then the corresponding surface is hyperbolic (indeed, the subgroup is cocompact Fuchsian). Thus an induced $k$--cycle for $k\ge 5$ corresponds to a convex subcomplex in the Davis complex of $W_\Gamma$ which is quasi-isometric to $\mathbb{H}^2$. Without additional conditions on the cycles, however, not all such surface subgroups are well-behaved.
	
	\begin{defn}We say that a subgraph $\Lambda\subseteq \Gamma$ is \emph{Morse} if
	\begin{itemize}
	\item $\Lambda$ is an induced subgraph, and
	\item Whenever $C$ is an induced $4$--cycle in $\Gamma$ with $C\cap \Lambda$ containing two non-adjacent vertices, then $C\subseteq \Lambda$
	\end{itemize}

Further, if $\Lambda$ is a $k$--cycle, we will say that $\Lambda$ is a Morse $k$--cycle.
 \end{defn}

\begin{figure}[h]
	\centering
\subfloat[Non-Morse]{\begin{tikzpicture}[scale=0.75,every node/.style={fill=black,circle,inner sep=0.1cm}]
\node(a) at (18:2){};
\node(b) at (90:2){};
\node(c) at (162:2){};
\node(d) at (234:2){};
\node(e) at (306:2){};
\node[fill=none, draw=black](f) at (0,0){};
\draw (a)--(b)--(c)--(d)--(e)--(a);
\draw[dashed] (a)--(f)--(c);
\end{tikzpicture}} \hspace{1cm} \subfloat[Morse]{\begin{tikzpicture}[scale=0.75,every node/.style={fill=black,circle,inner sep=0.1cm}]
	\node(a) at (18:2){};
	\node(b) at (90:2){};
	\node(c) at (162:2){};
	\node(d) at (234:2){};
	\node(e) at (306:2){};
	\node[fill=none, draw=black](f) at (0,0){};
	\draw (a)--(b)--(c)--(d)--(e)--(a);
	\draw[dashed] (a)--(f)--(b);
\end{tikzpicture}}
\caption{Non-Morse and Morse $5$--cycles}\label{fig:5cycles}	
\end{figure}
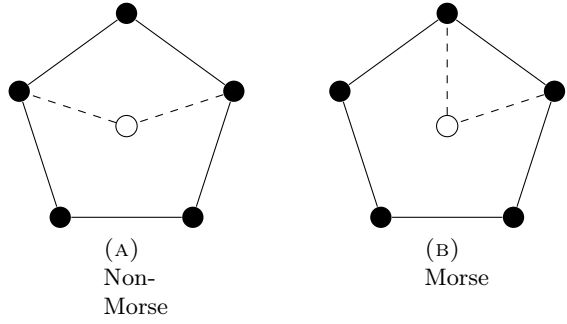

	This will be the key notion studied in the Subsection~\ref{subsec:groups} and Section~\ref{sec:proofs}. Note that in \cite{GKLS}, non-Morse cycles are called \emph{burst}.
	
	\subsection{Probability Notation and Background}
	
	By $\Gamma\in\G(n,p)$, we will mean that $\Gamma$ is random graph with $n$ vertices where two vertices are connected by an edge with probability $p$, independently of all other pairs. This is commonly referred to as the Erd\H{o}s--R\'enyi model of the random graphs, and $p$ is often called the \emph{density}.
	
	Given a function $p\colon \mathbb{N}\to (0,1)$ we will say that a property $\mathcal P$ of random graphs happens asymptotically almost surely (abbreviated a.a.s.) if the probability that $\Gamma\in\G(n,p)$ has the property $\mathcal P$ tends to $1$ as $n\to \infty$, or as we will often write:
	$$\lim_{n\to\infty}\Pb(\Gamma\in\mathcal P)=1.$$
	
	We use Landau notation. For functions $f,g\colon \mathbb{N}\to \mathbb{R}$, we say that:
	
	\begin{itemize}
		\item $f=O(g)$ if there is a constant $C$ so that $f(n)\le C\cdot g(n)$ for all $n$
		\item $f=o(g)$ if $\ds\lim_{n\to\infty} \frac{f(n)}{g(n)} = 0$
		\item $f=\Omega(g)$ if $g=O(f)$
		\item $f=\omega(g)$ if $g=o(f)$
		\item $f=\Theta(g)$ if $f=O(g)$ and $g=O(f)$
	\end{itemize}
	
	\subsection{Geometric Group Theory Background}\label{subsec:groups}

For any finitely generated group, $G=\gspan{S}$, with $S$ finite and inverse closed, we can construct the Cayley graph of $G$ with respect to $S$, denoted $\Cayley{G}{S}$, whose vertices correspond to elements of $g$ and $\{g,h\}$ span an edge of $\Cayley{G}{S}$ if and only if $gs=h$ for some $s\in S$. In this way, $G$ becomes a metric space, the distance between two elements is the length of shortest edge path in $\Cayley{G}{S}$ between them (called the \emph{word metric}). Changing generating sets leads to different, but similar, metrics.

\begin{defn}Let $\lambda \ge 1$ and $\epsilon\ge 0$. Given two metric space $(X, d_X)$ and $(Y, d_Y)$, a  $(\lambda,\epsilon)$--\emph{quasi-isometric} embedding is a function $f\colon X\to  Y$ so that for every $x_1, x_2\in X$:
	$$\frac{1}{\lambda}d_X(x_1, x_2)-\epsilon \le d_Y(f(x_1),f(x_2)) \le \lambda d_X(x_1,x_2)+\epsilon.$$
	
If further $d_Y(y,f(X))<\epsilon$ for all $y\in Y$, then $f$ is called a $(\lambda, \epsilon)$--quasi-isometry, and $X$ and $Y$ are \emph{quasi-isometric}.\end{defn}

It's well-known that if $S, S'$ are finite generating sets for a group $G$, then $\Cayley{G}{S}$ and $\Cayley{G}{S'}$ are quasi-isometric. By a $(\lambda,\epsilon)$--\emph{quasi-geodesic} (or just quasi-geodesic) in $Y$, we mean a quasi-isometric embedding of an interval.
	
Recall the following from the introduction. 
\begin{defn}Given a graph $\Gamma = (V,E)$, the \emph{right-angled Coxeter group} defined by $\Gamma$ is:
	$$W_\Gamma := \gspan{V \mid v^2=1\ \forall v\in V, vw=wv \iff \{v,w\}\in E}.$$
	That is, the generators for $W_\Gamma$ are in one-to-one correspondence with the vertices of $\Gamma$, each generator has order $2$, and two generators commute if and only if the corresponding vertices are adjacent in $\Gamma$.\end{defn}

Critically, $W_\Gamma\cong W_\Lambda$ if and only if $\Gamma\cong \Lambda$ \cite{Radcliffe:rigidity}. Further, if $\Lambda \subseteq \Gamma$ is an \emph{induced} subgraph, then $V(\Lambda)$ generates a subgroup isomorphic to $W_\Lambda$. Subgroups of this form are called \emph{special subgroups}.

Special subgroups of right-angled Coxeter groups are quasi-isometrically embedded, in fact, in the standard generating sets the embedding is isometric (and convex). The next set of definitions describes behavior in a metric space that is more restricted than quasi-isometric embedding that is motivated by the study of hyperbolic groups. The following notion of a \emph{Morse} (also called \emph{strongly quasiconvex}) subgroup was introduced by Tran in \cite{Tran:quasiconvex} and Genevois in \cite{Genevois:Hyperbolicities}.

\begin{defn} Let $N\colon [1,\infty)\times [0,\infty)\to [0,\infty)$ be a function (called a Morse gauge). Let $X$ be a metric space and $Y\subseteq X$. We say that $S$ is \emph{Morse} if the inclusion $Y\hookrightarrow X$ is a quasi-isometric embedding and for every $(\lambda,\epsilon)$--quasigeodesic $\gamma$ in $X$ with endpoints $x, y\in Y$, $\gamma$ stays in the $N(\lambda, \epsilon)$ neighborhood of $Y$.

If $X$ is a Cayley graph of a finitely generated group $G$, and $Y$ is the set of vertices corresponding to a subgroup $H$, we say that $H\le G$ is a \emph{Morse subgroup}. If $Y$ is a quasi-geodesic, we say that $Y$ is $N$--Morse.\end{defn}

Morse subgroups play an important role in understanding the space of ``hyperbolic--like" directions in a group. There is also a closely related notion of \emph{stable subgroups}, introduced by Durham and Taylor in \cite{DT:stable}. For finitely generated groups, a subgroup is stable if it is Morse and hyperbolic \cite{Tran:quasiconvex}. 

The connection between Morse geodesics and hyperbolic directions is formalized by the following construction and results due to Cordes \cite{Cordes:Morse}, inspired by similar work of Charney and Sultan \cite{CharneySultan}.

\begin{defn}[Morse Boundary] Given a proper metric space $X$, we will say that two geodesics $\alpha,\beta \colon [0,\infty)\to X$ are equivalent if there is a constant $k$ so that $d(\alpha(t),\beta(t))\le K$ for all $t$. We denote the equivalence class of $\beta$ by $[\beta]$.
	
Let $p\in X$ and let $N\colon [1,\infty)\times [0,\infty)\to [0,\infty)$. We set
	
	$$\partial_M^N(X)_p = \{[\beta]: \exists \alpha\in [\beta] \text{ with } \alpha\ \text{$N$--Morse, }\alpha(0)=p\}$$
We give $\partial_M^N(X)_p$ topology of uniform convergence of representatives on compact subsets of $X$. 

Then the \emph{Morse boundary} of $X$ (at $p$) is:
$$\partial_M(X)_p = \lim_{\longrightarrow} \partial_M^N(X)_p.$$
\end{defn}

\begin{prop}\cite{Cordes:Morse}\label{prop:Morse} Let $X,Y$ be proper geodesic metric spaces and $p,q\in X$. Then:
	
	\begin{enumerate}
		\item $\partial_M(X)_p \cong \partial_M(X)_q$
		\item if $f\colon X\to Y$ is a quasi-isometry, then $\partial_M(X)_p\cong\partial_M(Y)_{f(p)}$
		\item\cite{Tran:quasiconvex} if $Y\subseteq X$ is Morse, and $p\in Y$, then the inclusion of $Y$ in $X$ induces an inclusion:
		$$\iota_M\colon \partial_M(Y)_p \hookrightarrow \partial_M(X)_p.$$
	\end{enumerate}
 \end{prop}

For a right-angled Coxeter group, determining when a special subgroup is Morse is particularly simple.

 	\begin{prop}[\cite{Genevois:Hyperbolicities,RST:convexity}]\label{thm:Morsebdypentagon} Let $\Gamma$ be a simplicial graph and $\Lambda$ an induced subgraph of $\Gamma$. Then the special subgroup $W_\Lambda \le W_\Gamma$ is Morse if and only if $\Lambda$ is a Morse subgraph of $\Gamma$.\end{prop}

	In \cite{CCS:Morse} it is shown that the Morse boundary of any right-angled Artin group is totally disconnected. Using this, as well as  Propositiona~\ref{prop:Morse} and~\ref{thm:Morsebdypentagon}, we see the following.
	
	\begin{cor}\label{cor:qiRAAG} Suppose $\Gamma$ contains a Morse $k$--cycle for $k\ge 5$. Then $W_\Gamma$ is not quasi-isometric to a right-angled Artin group.
	
Further, if $W_\Gamma$ is one-ended and $\Gamma$ contains a Morse $4$--cycle, then $W_\Gamma$ is not quasi-isometric to a right-angled Artin group.\end{cor}
	
	\begin{proof} Let $C$ be a Morse $k$--cycle with $k\ge 5$. Then, since $C$ is induced, the special subgroup $W_C\le W_\Gamma$ is quasi-isometrically embedded and virtually a hyperbolic surface group. 	Further, by Proposition~\ref{thm:Morsebdypentagon} $W_C$ is Morse in $W_\Gamma$. Since $W_C$ is quasi-isometric to $\mathbb{H}^2$ by Proposition \ref{prop:Morse}$\partial_M(W_C)\cong \partial_M(\mathbb{H}^2)\cong S^1$.  
	
	Since $W_C$ is a Morse subgroup, by Proposition~\ref{prop:Morse}(2), $\partial_M(W_C)$ embeds in $\partial_M(W_\Gamma)$, and so $\partial_M(W_\Gamma)$ contains an embedded copy of $S^1$. Hence $\partial_M(W_\Gamma)$ is not totally disconnected. But in \cite{CCS:Morse} it is shown that the Morse boundary of any right-angled Artin group is totally disconnected. Thus, by Proposition~\ref{prop:Morse}(2) $W_\Gamma$ is not quasi-isometric to a right-angled Artin group.
		
	If $k=4$, then $W_C$ is virtually $\mathbb{Z}^2$ (thus non-hyperbolic) and Morse in $W_\Gamma$. Now suppose that $W_\Gamma$ is one-ended. Then, if it were quasi-isometric to a right-angled Artin group, that group must be one-ended as well. But in \cite[Corollary 7.4(d)]{RST:convexity} (as well as \cite{Genevois:Hyperbolicities, Tran:quasiconvex}) it is shown that the only Morse subsets of a one-ended right-angled Artin group are hyperbolic. Thus $W_\Gamma$ is not quasi-isometric to a right-angled Artin group.
	
	\end{proof}

		We should note that Morse $k$--cycles with $k\ge 5$ are not the only way that the Morse boundary of a right-angled Coxeter groups could contain a copy of $S^1$, though examples from \cite{GKLS} use Morse cycles in a finite index subgroup. In unpublished work, Tran has claimed that if $\Gamma$ is the 1--skeleton of a cube, then the Morse boundary of $W_\Gamma$ contains an embedded copy of $S^1$, but no finite index reflection subgroup has a Morse cycle in its defining graph. Further work has been done on this by Russell, Spriano and Tran, as well as Karrer, attempting to classify Morse subgroups of a right-angled Coxeter groups and give conditions under which we can guarantee that the Morse boundary is totally disconnected. 

\section{Proofs of the Main Theorems}\label{sec:proofs}

The following provides a threshold for the existence of Morse pentagons. The proof is similar to standard subgraph inclusion proofs in the Erd\H{o}s--R\'enyi model, though extra care must be taken when estimating the second moment: the strength of the stability condition makes even disjoint sets of vertices not quite independent (though very close) and so more care must be taken than in the standard argument (see \textit{e.g.}, \cite[Chapter 4]{AlonSpencer}).
 
Before the proof of the main Theorems, the following, technical Lemma will be helpful.

\begin{lem}\label{lem:linkdependent} Let $\Gamma\in\G(n,p)$ and let $v, w_1, w_2, w_3, w_4 \in \Gamma$. Then:
	
\begin{enumerate}
	\item $\Pb[v\in\link(w_1)\mid v\not\in\left(\link(w_1)\cap\link(w_2)\right)\cup\left(\link(w_1)\cap\link(w_3)\right)] = \frac{p(p-1)}{p^2-p-1} = p+o(p)$;
	\item $\Pb[v\in\link(w_1)\cap \link(w_2)\mid v\not\in \link(w_2)\cap\link(w_3)] = \frac{p^2}{1+p} = p^2+o(p^2)$;
	\item $\Pb[v\in \link(w_1)\cap\link(w_2)\mid v\not\in \left(\link(w_1)\cap\link(w_3)\right)\cup\left(\link(w_2)\cap\link(w_4)\right)] = \frac{p^2}{(1+p)^2} = p^2+o(p^2).$
\end{enumerate}
\end{lem}

\begin{proof} We prove (2), the remaining proofs are similar.
Fix $v\in \Gamma$ and let $L_v(x,y)$ be the event that $v\in\link(x)\cap\link (y)$. We use the following, familiar formula from elementary probability:

$$\Pb[L_v(w_1,w_2)\mid \neg L_v(w_2, w_3)] = \frac{\Pb[L_v(w_1,w_2)\land\left(\neg L_v(w_2, w_3)\right)]}{\Pb[\neg L_v(w_2, w_3)]}.$$

Clearly, $\Pb[\neg L_v(w_2, w_3)] = 1-p^2$. Now, if for $L_v(w_1, w_2)\land \left(\neg L_v(w_2, w_3)\right)$ to be true, we must have that $v$ is joined by an edge to $w_1$ and $w_2$, but not $w_3$. This occurs with probability $p^2(1-p)$.

$$\Pb[L_v(w_1,w_2)\mid \neg L_v(w_2, w_3)] = \frac{p^2(1-p)}{1-p^2}= \frac{p^2}{1+p} = p^2+o(p^2).$$
 
\end{proof}

  \begin{thm}\label{thm:pentagons} Let $\lambda = \sqrt{\ds\frac{1}{2}}$ and let $\epsilon>0$. If $np\to\infty$ and $p(n)<(\lambda - \epsilon) \sqrt{\ds\frac{\log{n}}{n}}$, then a random graph at density $p(n)$ a.a.s. contains a Morse pentagon.\end{thm}

\begin{proof} 
	Let $X$ be the number of Morse $5$--cycles in $\Gamma$. Then $X=\sum_{S\subset V(\Gamma), |S|=5|} X_S$, where $$X_S = \begin{cases} 1, &\text{ if } S \text{ forms a Morse $5$--cycle}\\
		0, & \text{ otherwise}\end{cases}.$$
	 
	We use the second moment method to show that $\Pb[X>0]\to 1$ as $n\to \infty$. To do this, first we show that $\E[X]\to \infty$.
		
Let $S$ be a set of $5$ vertices. To form a Morse pentagon, $S$ must form an induced pentagon, and no two non-adjacent vertices in $S$ can have a common neighbor outside of the $5$--cycle (see Figure~\ref{fig:5cycles}).

The probability that $S$ forms an induced $5$--cycle is $\frac{5!}{10}p^5(1-p)^5$. Given that $S$ forms an induced $5$--cycle, the probability that a vertex outside of $S$ has two non-adjacent neighbors in $S$ is $5p^2+o(p^2)$, and thus the probability that $S$ is Morse is $\frac{5!}{10}p^5(1-p)^5(1-5p^2+o(p^2))^{n-5}>C p^5(1-p)^5e^{-5p^2n}$ for $n$ sufficiently large.

Thus:

$$\E[X] > C\binom{n}{5}p^5(1-p)^5e^{-5p^2n}=\Omega((np)^5 e^{-5p^2n}).$$

If $p(n) = O(n^{-1/2})$, and $p(n) = \omega(n^{-1})$, then $(np)^5\to \infty$ and $e^{-5p^2n}=O(1)$.

If $p(n)=\Omega(n^{-1/2})$ and $p(n) < (\lambda - \epsilon) \sqrt{\ds\frac{\log{n}}{n}}$, then $$\E[X] = \Omega(n^{5/2}e^{(-5/2+5\epsilon)\log{n}}) = \Omega(n^{5\epsilon}).$$

In both cases, we see that $\E[X]\to\infty$.

Now we estimate $\text{Var}(X) = \E[(X-\E[X])^2]=\E[X^2]-E[X]^2$, and show that $\text{Var}(X) = o(\E[X]^2)$. From there, the result follows immediately from Chebyshev's inequality.

Since $X$ is a sum of indicator variables, we see that:
$$\text{Var}(X) \le \E[X] + \sum_{S\neq S'}\left(\E[X_SX_{S'}] - \E[X_S]\E[X_{S'}]\right)$$

But $\E[X_S]\E[X_{S'}] \ge 0$, so we may (selectively) remove those terms from the sum. We estimate:
   
$$\text{Var}(X) \le \E[X] + \sum_{S\cap S' = \emptyset}\left(\E[X_SX_{S'}] - \E[X_S]\E[X_S']\right) + \sum_{S\cap S'\neq\emptyset} \E[X_SX_{S'}].$$

To complete the proof, we show that each of these sums is $o(\E[X]^2)$.

Noting that $\E[X_SX_S'] = \Pb[X_S=1] \Pb[X_{S'}=1\mid X_S = 1]$, we estimate the conditional probability for different configurations of $S, S'$.

If $S\cap S' = \emptyset$, then the probability that $S'$ forms an induced $5$--cycle is independent of $S$. Now the probability that no vertex outside $S'$ has two common neighbors which are not adjacent in $S'$ is at most the probability of the same event for all vertices outside $S\cup S'$. Thus, the probability that $S'$ is a Morse $5$--cycle is at most $\frac{5!}{10}p^5 (1-p)^5 (1-5p^2)^{n-10}$, and hence:
$$\E[X_SX_{S'}] \le \left(\frac{5!}{10}\right)^2p^{10}(1-p)^{10}(1-5p^2)^{n-5}(1-5p^2)^{n-10}$$

and thus:

\begin{align*}\sum_{S\cap S' = \emptyset}\left(\E[X_SX_{S'}] - \E[X_S]\E[X_S']\right)&\le \binom{n}{5}\binom{n-5}{5}\left(\frac{5!}{10}\right)^2p^{10}(1-p)^{10}(1-5p^2)^{2(n-5)}\left((1-5p^2)^{-5}-1\right)\\ &\le  \E[X]^2\left((1-5p^2)^{-5}-1\right) = o(\E[X]^2).\end{align*}

The last equality follows since $\left((1-5p^2)^{-5}-1\right)=o(1)$, as $p=o(1)$.

Now suppose that $S\cap S' \neq \emptyset$. We break this down depending on the size of $S\cap S'$

\underline{Case 1.} If $|S\cap S'|=1$, let $S\cap S'=\{w\}$. By Lemma~\ref{lem:linkdependent}(1),  the probability $v\in\link(w)$ given that $S$ is not Morse (and so $w\not\in \link(w)\cap\link(w'))\cup (\link(w)\cap\link(w''))$ for $w',w''\in S$  is $p+o(p)$.
Thus:
$\Pb[X_{S'}=1\mid X_S =1] \le 5p^5(1-p)^5(1-3p^2-2(p(p+o(p))))^{n-9}$. Thus 
\begin{align*}\sum_{|S\cap S'|=1}\E[X_SX_{S'}] &\le \binom{n}{5}\binom{n-5}{4}\left(\frac{5!}{10}\right)^2p^{10}(1-p)^{10}\cdot 5(1-5p^2+o(p^2))^{2(n-9)}\\
	 &=O(n^9p^{10}(1-5p^2)^{2(n-5)}) = o(\E[X]^2)\end{align*}
 
 \underline{Case 2.} If $|S\cap S'|=2$, let $S\cap S'=\{v,w\}$.
 
 If $\{v,w\}\in E(\Gamma)$ then similar to Case 1, $\Pb[X_{S'}=1\mid X_S=1] \le 6p^4(1-p)^5(1-5p^2+o(p^2))^{n-8}$
Summing over all such $S, S'$ yields at most: 
$$5\binom{n}{5}\binom{n-5}{3}\cdot \frac{5!}{10}\cdot 6 p^9(1-p)^{10}(1-5(p^2+o(p^2))^{2(n-8)} O(n^8p^9(1-5p^2+o(p^2))^{2(n-5)})=o(\E[X]^2),$$

since $np\to \infty$.

If $\{v,w\}\not\in E(\Gamma)$, given that $S$ forms a Morse pentagon, the distance between $v$ and $w$ in $\Gamma$ is precisely $2$. Thus, $v,w$ cannot be adjacent in $S'$ and also cannot be distance $2$ in $S'$ (since the edge path in $S'$ joining them will be outside of $S$, see Figure~\ref{fig:overlap}). Thus $\Pb[X_{S'}=1 \mid X_S=1]=0.$

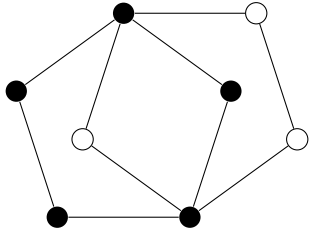
\begin{figure}
	\centering
	\begin{tikzpicture}[scale=0.75,every node/.style={fill=black,circle,inner sep=0.1cm}]
		\node(a) at (18:2){};
		\node(b) at (90:2){};
		\node(c) at (162:2){};
		\node(d) at (234:2){};
		\node(e) at (306:2){};
		\node[fill=none, draw=black](f) at (-0.7265,-0.2361){};
		\node[fill=none, draw=black](g) at (3.0777,-0.2361){};
		\node[fill=none, draw=black](h) at (2.3511,2){};
		\draw (a)--(b)--(c)--(d)--(e)--(a);
		\draw (b)--(f)--(e)--(g)--(h)--(b);
		
	\end{tikzpicture}
\caption{Two pentagons with intersection a pair of non-adjacent vertices. Vertices in $S$ are filled in. Neither pentagon can be Morse.}\label{fig:overlap}
\end{figure}
\underline{Case 3.} If $|S\cap S'|=3$, the analysis of the previous case shows that $\Pb[X_{S}'=1\mid X_S=1] = 0$ whenever $S'$ contains two non-adjacent vertices of $S$ but not the vertex between them. Thus, the only possibility contributing to our sum is when $S\cap S'$ forms an edge path of length $2$ in $S$.

If this case $\Pb[X_{S'}=1\mid X_S=1] \le 5p^3(1-p)^4(1-4p^2+o(p^2))^{n-7}$, since we already know that there is no vertex outside of $S\cup S'$ adjacent to both endpoints of the edge path, as $S'$ is Morse. Summing over all such $S, S'$ yields:

\begin{align*}\sum_{|S\cap S'|=3} \E[X_SX_{S'}] &\le \binom{n}{5}\binom{n-5}{2}\left(\frac{5!}{10}\right)\cdot 5 p^8(1-p^9)(1-5p^2)^{n-5}(1-4p^2+o(p^2))^{n-7}\\
	& = O(n^7p^8(1-5p^2)^{n-5}(1-4p^2)^{n-5})\end{align*}

Now, since $p<(\lambda-\epsilon)\sqrt{\frac{\log{n}}{n}}<0.9\sqrt{\frac{\log{n}}{n}}$:

\begin{align*}\left(\frac{1-4p^2}{1-5p^2}\right)^{n-5} &= \left(1+\frac{p^2}{1-5p^2}\right)^{n-5}\\
	&\le \left(1+p^2\right)^{n-5} \le (1+0.81\frac{\log{n}}{n})^{n-5}\\
	&\le n^{0.81}\end{align*}

Thus: $$\sum_{|S\cap S'|=3} \E[X_SX_{S'}]= O(n^7p^8(1-5p^2)^{n-5}(1-4p^2)^{n-5}) = O\left(\frac{\E[X]^2}{n^3p^2}\cdot\left(\frac{1-4p^2}{1-5p^2}\right)^{n-5}\right) = o(\E[X]^2),$$

since $np\to\infty$ implies that $\frac{1}{n^3p^2} =o(\frac{1}{n})$.

\underline{Case 4.} If $|S\cap S'| = 4$, then $S\cap S'$ must contain two non-adjacent vertices of $S$ but not the vertex between them. Hence $\Pb[X_{S'}=1\mid X_S=1]=0$.

This completes the case work. Putting the cases together yields:

$$\text{Var}(X) \le \E[X] + \sum_{S\cap S' = \emptyset} \left(\E[X_SX_{S'}]-\E[X_S]\E[X_S']\right) + \sum_{i=1}^3\sum_{|S\cap S'|=i} \E[X_SX_{S'}]=\E[X] + o(\E[X]^2)= o(\E[X]^2).$$

Applying Chebyshev's inequality, $\Pb[X>0] \to 1$ as $n\to\infty$. Thus, a.a.s. a random graph at density $p(n)< (\lambda-\epsilon)\sqrt{\frac{\log{n}}{n}}$ and $np\to\infty$ contains a Morse $5$--cycle.
\end{proof}
\begin{thm}\label{thm:noRAAG} Let $\lambda=\sqrt{\ds\frac{1}{2}}$ and let $\epsilon>0 $. If $np\to \infty$ and  $p(n)<(\lambda-\epsilon)\sqrt{\ds\frac{\log{n}}{n}}$, then a.a.s. the Morse boundary of $W_\Gamma$ contains an embedded copy of $S^1$, in particularly it is not totally disconnected. \end{thm}

\begin{proof} This follows immediately from Theorem~\ref{thm:pentagons}, Proposition~\ref{thm:Morsebdypentagon} and Proposition~\ref{prop:Morse}(3).\end{proof}

%
%
%
%
%
%
%
%
%
%
%

We next want to show that at densities above $(\lambda+\epsilon)\sqrt{\frac{\log{n}}{n}}$, there are no Morse cycles. It's straight forward to see that every cycle has at least one pair of vertices at distance $2$ with a common neighbor outside of the cycle; however, we must show that the set of all common neighbors does not form a clique in order to obtain an \emph{induced} $4$--cycle containing those two vertices. 
	
Recall that for a vertex $v\in \Gamma$, $\link(v)$ is the subgraph induced by the set of all vertices adjacent to $v$.

\begin{prop}\label{prop:unstable}Let $\lambda = \sqrt{\ds\frac{1}{2}}$ and let $\epsilon >0$. If $p(n) > (\lambda + \epsilon) \sqrt{\ds\frac{\log{n}}{n}}$ and then a.a.s. $\Gamma\in \G(n,p)$ contains no Morse $k$--cycle with $k\ge 5$.\end{prop}

\begin{proof} We must prove that either $\Gamma$ contains no induced $k$--cycles or each $k$--cycle $C\subseteq \Gamma$ has the property that there are non-adjacent vertices $v,w\in C$ with $\link(v)\cap\link(w)$ not a clique.
	
We first cover densities $p(n)=\omega(\sqrt\frac{\log{n}}{n})$, where for any two vertices $v,w$, $|\link(v)\cap\link(w)|$ is well-controlled. Indeed, it follows from Chernoff's inequalities that when $p(n)=\omega(\sqrt{\frac{\log{n}}{n}})$, $|\link(v)\cap\link(w)|=O(np^2)$ for any two vertices $v,w\in \Gamma$ \cite[Corollary 4.2]{BFRHS}. We now break down further, distinguishing cases where $p(n)\to 1$ as $n\to \infty$.

\underline{Case 1.} $p=\omega(\sqrt\frac{\log{n}}{n})$, and $p\le 1-\delta$ for some constant $\delta >0$. Here, the size of largest clique is $O(\log{n})=o(np^2)$ \cite[Lemma 4.3]{BFRHS};

\underline{Case 2.} $p\to 1$ as $n\to\infty$ and $p \le 1-\frac{\log{n}}{2n}$. In this case, the size of the common neighbor set is $O(n)$, but the size of the largest clique is $o(n)$.

In both of these cases, we see that no special subgroup arising from a cycle can be Morse, since every pair of vertices $v,w$ has $\link(v)\cap\link(w)$ larger than the size of the largest clique. Now, there are two more cases for $p$ very close to $1$.

\underline{Case 3.} If $p> 1- \frac{\log{n}}{2n}$ but $(1-p)n^2\to \infty$, then by \cite[Proof of Theorem 4.4, Case 3]{BFRHS}, $\Gamma$ decomposes as a non-trivial join. Hence, any set of non-adjacent vertices in $\Gamma$ are opposite vertices in a square.

\underline{Case 4.} If $(1-p)n^2\to\alpha<\infty$, then by \cite[Theorem V]{BHS:Coxeter} every vertex of $\Gamma$ a.a.s. has at most one non-neighbor ruling out an induced $k$--cycles for $k\ge 5$.

We now consider the case where $p=(\lambda+\epsilon)\sqrt{\frac{\log{n}}{n}}$ for a constant $\epsilon>0$. Note that in this case the size of the largest clique in $\Gamma$ is a.a.s. $5$ \cite[Chapter 4]{AlonSpencer}. For the remainder of the argument, we will condition on this. Before we continue the argument, we first claim that a.a.s. there are no very large induced cycles.

\underline{Claim.} For $p=(\lambda+\epsilon)\sqrt{\frac{\log{n}}{n}}$, a.a.s. a random graph $\Gamma$ at density $p$ contains no induced $k$--cycles for $k>3\sqrt{n\log{n}}$.

\begin{proof}[Proof of Claim.] By Markov's inequality, the probability that $\Gamma$ contains an induced $k$--cycle is at most:
	
	$$\binom{n}{k}\frac{k!}{2k}p^k(1-p)^{\binom{k}{2}-k} \le (np)^k(1-p)^{\frac{k^2-3k}{2}}\le (np)^ke^{-p\frac{k^2-3k}{2}}$$
	
	Since $p=(\lambda+\epsilon)\sqrt{\frac{\log{n}}{n}}$, this becomes
	
	$$(\lambda+\epsilon)^k(n\log{n})^{k/2}n^{-(\lambda+\epsilon)\frac{k^2-3k}{2\sqrt{n\log{n}}}}$$
			
	But, for $k>3\sqrt{n\log{n}}$, $-(\lambda+\epsilon)\frac{k^2-3k}{2\sqrt{n\log{n}}}+\frac{k}{2}<-\frac{3(\lambda+\epsilon)k}{2}$. Thus, the probability that $\Gamma$ contains an induced $k$--cycle for any $k>3\sqrt{n\log{n}}$ is atmost:
		
	$$\sum_{k=3\sqrt{n\log{n}}}^n \left(\frac{(\lambda+\epsilon)\sqrt{\log{n}}}{n^{3(\lambda+\epsilon)/2}}\right)^k=\sum_{k=3\sqrt{n\log{n}}}^n o(n^{-3k/2}) = o(1).$$ 

Thus, a.a.s., $\Gamma$ does not contain any induced $k$--cycles with $k>3\sqrt{n\log{n}}.$\end{proof}

Let $\{v,w\}$ be a pair of vertices at distance $2$ in an induced cycle $C$ with $|C|=k$ of $\Gamma$ (note that, by the claim, we may assume that $5\le k \le 3\sqrt{n\log{n}}$). Let $E(v,w)$ be the event that $\link(v)\cap\link(w)$ is a clique. Note that $|\link(v)\cap\link(w)|$ outside of $C$ is a binomial variable with $n-k$ trials and success probability $p^2$, and $\{v,w\}$ already have one common neighbor (coming from $C$). Thus the probability that $\link(v)\cap\link(w)$ forms a clique (given that there are no cliques of size at least $6$) is:  

$$\sum_{l=0}^4\binom{n-k}{l}p^{2l}(1-p^2)^{n-k-l}\cdot p^{\binom{l+1}{2}}.$$

Since $p(n)=(\lambda+\epsilon)\sqrt{\frac{\log{n}}{n}}$, this is $\Theta(n^{-(\lambda+\epsilon)^2})$. 


Let $C$ be a $k$--cycle for $5\le k\le 3\sqrt{n\log n}$ with vertex set $\{v_1, v_2, \ldots, v_k\}$. For a fixed $1\le i\le k-2$, we ignore all vertices in $C\cup \bigcup_{j=1}^{i-1} \left(\link(v_j)\cap \link(v_{j+2})\right)$. Conditioning on $\bigwedge_{j=1}^{i-1}E(v_j, v_{j+2})$, we see that $|\link(v_j)\cap \link(v_{j+2})|\le 5$. By Lemma~\ref{lem:linkdependent}(1), the probability that $v$ is adjacent to $v_{i+1}$, conditional on $v\not\in\link(v_{i-1})\cap\link(v_{i+1})$ is $p+o(p)$, and similarly the probabality that $v\in\link(v_i)\cap\link(v_{i+2})$ given that $v\not\in\link(v_{i-2})\cap\link(v_i)$ is $p(p+o(p))=p^2+o(p^2)$. Thus we see the following conditional probability for $E(v_i, v_{i+2})$: 

\begin{align*}&\Pb\left[E(v_i v_{i+2})\Bigg|\bigwedge_{j=1}^{i-1}E(v_j, v_{j+2})\right]\le\\
	 &\sum_{l=0}^4\binom{n-k-5(i-1)}{l}(p^2+o(p^2))^{2l}(1-p^2+o(p^2))^{n-k-5(i-1)-l}\cdot (p+o(p))^lp^{\binom{l}{2}}=O(n^{-(\lambda+\epsilon)^2}),&
\end{align*}

since $i\le k=o(n)$.

Now, for $i=k-1$, using Lemma~\ref{lem:linkdependent}(3) the probability that $v\in \link(v_{k-1})\cap\link(v_1)$ given that $v\not\in \left(\link(v_{k-3})\cap\link(v_{k-1})\right)\cup\left(\link(v_1)\cap\link(v_3)\right)$ is $p^2+o(p^2)$. By the same computation:
$$\Pb\left[E(v_{k-1}, v_{1})\Bigg| \bigwedge_{j=1}^{k-2}E(v_j, v_{j+2})\right]=O(n^{-(\lambda+\epsilon)^2}).$$

Finally, for $i=k$ we again use Lemma~\ref{lem:linkdependent}(3) to see that the probability that $v\in \link(v_k)\cap \link(v_2)$ conditional on $v\not\in\left(\link(v_{k-2})\cap \link(v_k)\right)\cup\left(\link(v_2)\cap \link(v_4)\right)$ us $p^2+o(p)$. Furthe by Lemma~\ref{lem:linkdependent}(1), the probability that $v$ is adjacent to $v_1$ given that $v\not\in(\link(v_{k-1})\cap \link(v_1))\cup (\link(v_1)\cap \link(v_3))$ is $p+o(p)$.

Thus, by the same computation as above:

$$\Pb\left[E(v_{k}, v_{2})\Bigg| \bigwedge_{j=1}^{k-2}E(v_j, v_{j+1})\land E(v_{k-1},v_1)\right]=O(n^{-(\lambda+\epsilon)^2}).$$
Thus:
$$\Pb\left[\bigwedge_{i=1}^{k-2} E(v_i,v_{i+2})\land E(v_{k-1}, v_1)\land E(v_k, v_2)\right] = O(n^{-k(\lambda+\epsilon)^2}) = O(n^{-k/2-k(\sqrt{2}\epsilon - \epsilon^2)})$$

Thus, the expected number of Morse$k$--cycles (with $k\le 3\sqrt{n\log{n}}$) is at most:

$$\sum_{k=5}^{3\sqrt{n\log{n}}} O((np)^kn^{-k(\lambda+\epsilon)^2})$$

With $p=(\lambda+\epsilon)\sqrt{\frac{\log{n}}{n}}$ for $\epsilon>0$ this is $$\sum_{k=5}^{2\sqrt{n\log{n}}} O(n^{-k(\sqrt{2}\epsilon+\epsilon^2)}\log^{k/2}{n})=o(1).$$ Thus, by Markov's inequality, 
a.a.s. no such $k$--cycles exist. Hence, the corresponding special surface subgroups are not Morse.
\end{proof}

The following Corollary builds on this, relying on \cite{GLR:surfacesubgroups}.

\begin{cor}\label{cor:specialstable}Let $\epsilon>0, \lambda=\sqrt{\ds\frac{1}{2}}$ and let $p(n)>(\lambda+\epsilon)\sqrt{\frac{\log{n}}{n}}$. Then for $\Gamma\in\G(n,p)$, a.a.s. every hyperbolic Morse special subgroup in $W_\Gamma$ is virtually free.\end{cor}

\begin{proof} Let $\Lambda$ be a subgraph of $\Gamma$ so that $W_\Lambda$ is  hyperbolic. By \cite{GLR:surfacesubgroups}, $W_\Lambda$ is either virtually free or $\Lambda$ contains an induced $k$--cycle for $k\ge 4$. We rule out the latter possiblity.
	
Since $W_\Lambda$ is hyperbolic, $\Lambda$ cannot contain an induced $4$--cycle, so $W_\Lambda$ contains an induced $k$--cycle, $C$, with $k\ge 5$. Further, since $\Lambda$ is square-free, no two non-adjacent vertices of $C$ lie in a common join and so by Propositon~\ref{thm:Morsebdypentagon}, the corresponding subgroup $W_C$ is Morse in $W_\Lambda$.

But, by Proposition~\ref{prop:unstable}, if $\Gamma\in\G(n,p)$, $W_C$ is a.a.s. not Morse in $W_\Gamma$. Hence, $W_\Lambda$ is not Morse either. Thus $\Lambda$ cannot contain a $k$--cycle with $k\ge 5$, and so $W_\Lambda$ is virtually free.

\end{proof}

This, however, does not fully resolve the question of whether the Morse boundary of a random right-angled Coxeter groups with $p(n)>(\lambda+\epsilon)\sqrt{\frac{\log{n}}{n}}$ is a.a.s. totally disconnected. A promising path forward is to look at Karrer's property $\mathcal{C}$ introduced in \cite{Karrer:PropC}, which implies totally disconnected Morse boundary

\subsection*{Morse $4$--cycles}
While $k$--cycles for $k\ge 5$ were primarily studied here, the same can be done for $k=4$. In this case, the threshold is somewhat different, since there are only two distinct sets of vertices at distance $2$ in the cycles. Therefore, the expected number of Morse $4$--cycle is:
$$\binom{n}{4}\frac{4!}{8}p^4(1-p)^2(1-2p)^{n-4}=O((np)^4e^{-2p^2n})$$
which tends to $\infty$ when $p(n)<(1-\epsilon)\sqrt{\frac{\log{n}}{n}}$. Indeed, copying the proof of Theorem~\ref{thm:pentagons} \emph{mutatis mutandis}, we see:

\begin{prop}\label{prop:Morsesquare} Let $\epsilon>0$, $p(n)<(1-\epsilon)\sqrt{\frac{\log{n}}{n}}$ and $np\to\infty$, then a.a.s. $\Gamma\in\G(n,p)$ contains a Morse $4$--cycle. In particular, $W_\Gamma$ contains a Morse subgroup isomorphic to $\mathbb{Z}^2$.\end{prop}
	
This immediately leads to the following Corollary, using Corollary~\ref{cor:qiRAAG}, Theorem~\ref{thm:pentagons} and Proposition~\ref{prop:Morsesquare}.

\begin{cor}\label{cor:noRAAG}Let $\epsilon>0$ if $p(n)<\left(1-\epsilon\right)\sqrt{\ds\frac{\log{n}}{n}}$ and $np\to\infty$, then a.a.s. a right-angled Coxeter group at density $p(n)$ is not quasi-isometric to a right-angled Artin group.\end{cor}

Now, since no two opposite vertices of a Morse $4$--cycle have a common neighbor, Morse $4$--cycles represent an isolated vertices in $\square(\Gamma)$. Thus, we obtain the following corollary, which complements \cite{BFRS}.

\begin{cor}\label{cor:disconnect}Let $\epsilon>0$, $p(n)<(1-\epsilon)\sqrt{\frac{\log{n}}{n}}$ and $np\to\infty$, then for $\Gamma\in\G(n,p)$ a.a.s. $\square(\Gamma)$ contains an isolated vertex. In particular, $\square(\Gamma)$ is disconnected.\end{cor}
	
Indeed, above this threshold, a simple Markov's inequality shows that $\square(\Gamma)$ a.a.s. has no isolated vertices. This leads to the following conjecture, on connectedness of $\square(\Gamma)$, analogous to the classical proof by Erd\H{o}s and R\'enyi of the connectedness threshold for random graphs \cite{ErdosRenyi1}. 

\begin{conj}\label{conj:squareconnect} Let $\epsilon>0$ and $p(n)>(1+\epsilon)\sqrt{\frac{\log{n}}{n}}$, then for $\Gamma\in\G(n,p)$ a.a.s. $\square(\Gamma)$ is connected.\end{conj}

\bibliography{Random_RACG_Morse}
\bibliographystyle{amsalpha}
		
\end{document}